\documentclass[letterpaper, 10pt, conference]{ieeeconf}      % Use this line for a4
\IEEEoverridecommandlockouts                              % This command is only
\usepackage{amsmath,mathrsfs,amsfonts,amssymb,graphicx,epsfig}
 \usepackage{amsthm}
\usepackage{subcaption}
\usepackage{color,multirow,rotating}
\usepackage{algorithm,algpseudocode,algorithmicx}
\usepackage{cite}
\usepackage{url}
\usepackage{framed}
\usepackage{bm}

\newtheorem{theorem}{Theorem}

\newtheorem{lemma}{Lemma}
\newtheorem{proposition}{Proposition}
\newtheorem{remark}{Remark}

\usepackage{tikz}
\usetikzlibrary{calc,trees,positioning,arrows,chains,shapes.geometric,decorations.pathreplacing,decorations.pathmorphing,shapes, matrix,shapes.symbols}

\newcommand{\bx}{{\bm{x}}}
\newcommand{\by}{{\bm{y}}}
\newcommand{\differential}{{\rm{d}}}
\renewcommand{\det}{{\mathrm{det}}}

\newcommand{\vol}{{\mathrm{vol}}}

\allowdisplaybreaks

\title{\LARGE\textbf{
The Convex Geometry of Integrator Reach Sets}
}

\author{Shadi Haddad, and Abhishek Halder% <-this % stops a space
\thanks{Shadi Haddad, and Abhishek Halder are with the Department of Applied Mathematics, University of California, Santa Cruz, CA 95064, USA,
        {\tt\small{\{shhaddad,ahalder\}@ucsc.edu}}%
}}

\begin{document}

\maketitle
\thispagestyle{empty}
\pagestyle{empty}

%%%%%%%%%%%%%%%%%%%%%%%%%%%%%%%%%%%%%%%%%%%%%%%%%%%%%%%%%%%%%%%%%%%%%%%%%%%%%%%%
\begin{abstract}
We study the convex geometry of the forward reach sets for integrator dynamics in finite dimensions with bounded control. We derive closed-form expressions for the volume and the diameter (i.e., maximal width) of these sets in terms of the state space dimension, control bound, and time. These results are novel, and use convex analysis to give an analytical handle on the ``size" of the integrator reach set. Several concrete examples are provided to illustrate our results. We envision that the ideas presented here will motivate further theoretical and algorithmic development in reach set computation. 
\end{abstract}

%%%%%%%%%%%%%%%%%%%%%%%%%%%%%%%%%%%%%%%%%%%%%%%%%%%%%%%%%%%%%%%%%%%%%%%%%%%%%%%%

\section{Introduction}
We consider the $d$-dimensional integrator dynamics
\begin{align}
\dot{\bm{x}} = \bm{A}\bm{x} + \bm{b}u, \quad\bm{x}\in\mathbb{R}^{d}, \quad u\in[-\mu,\mu],
\label{IntegratorDyn}	
\end{align}
with given $\mu>0$, $d = 2, 3, \hdots$, and
\begin{align}
\bm{A} := \left[\bm{0} \mid \bm{e}_{1} \mid \bm{e}_{2} \mid \hdots \mid \bm{e}_{d-1} \right], \quad \bm{b} := \bm{e}_{d},
\label{StateMatrices}	
\end{align}
where $\bm{0}$ denotes the $d\times 1$ column vector of zeros, and $\bm{e}_{i}$ is the $i$-th basis (column) vector in $\mathbb{R}^{d}$ for $i=1,\hdots,d$. Our intent in this paper is to study the geometry of the \emph{forward reach set} $\mathcal{R}\left(\mathcal{X}_{0},t\right)$ for (\ref{IntegratorDyn}) at time $t$, starting from a given compact convex set of initial conditions $\mathcal{X}_{0}\subset\mathbb{R}^{d}$, i.e.,
\begin{align}
\mathcal{R}\left(\mathcal{X}_{0},t\right) := \big\{\bm{x}(t)\in\mathbb{R}^{d} \mid \dot{\bm{x}} = &\bm{A}\bm{x} + \bm{b}u, \quad \bm{x}(0) \in\mathcal{X}_{0},\nonumber\\
&u\in[-\mu,\mu]\big\}.
\label{DefReachSet}	
\end{align}
In words, $\mathcal{R}\left(\mathcal{X}_{0},t\right)$ is the set of all states the controlled dynamics (\ref{IntegratorDyn}) can reach at time $t>0$, starting from the set $\mathcal{X}_{0}$ at $t=0$, with bounded control $u(t)\in[-\mu,\mu]$. Formally,
\begin{align}
\!\!\mathcal{R}\left(\mathcal{X}_{0},t\right)\! &= \exp(t\bm{A})\mathcal{X}_{0} \dotplus \!\!\int_{0}^{t}\!\!\!\exp\left((t-\tau)\bm{A}\right)\bm{b}\left[-\mu,\mu\right]\differential\tau\nonumber\\
&= \exp(t\bm{A})\mathcal{X}_{0} \dotplus \!\!\int_{0}^{t}\!\!\!\exp\left(s\bm{A}\right)\bm{b}\left[-\mu,\mu\right]\differential s,
\label{SetValuedIntegral}	
\end{align}
where $\dotplus$ denotes the Minkowski sum, and the set-valued Aumann integral \cite{aumann1965integrals} above is defined as follows. For any point-to-set function $F(\cdot)$, we define 
\begin{align}
\int_{0}^{t}F(s)\differential s := \lim_{\Delta\downarrow 0}\:\sum_{i=0}^{\lfloor\frac{t}{\Delta}\rfloor}\Delta F(i\Delta),
\label{DefIntegralOfSetValuedFn}
\end{align}
where the summation symbol $\Sigma$ stands for the Minkowski sum, and $\lfloor\cdot\rfloor$ is the floor operator; see e.g., \cite{varaiya2000reach}. We will often consider the special case of singleton initial set $\mathcal{X}_{0}\equiv\{\bm{x}_{0}\}$. 

It is straightforward to prove that $\mathcal{R}\left(\mathcal{X}_{0},t\right)$ is a compact convex subset of $\mathbb{R}^{d}$ for all $t>0$. Notice however, that the (space-time) \emph{forward reachable tube} 
\begin{align}
\overline{\mathcal{R}}\left(\mathcal{X}_{0},t\right) := \bigcup_{0\leq \tau \leq t} \mathcal{R}\left(\mathcal{X}_{0},\tau\right),
\label{DefReachableTube}	
\end{align}
need not be convex in $\mathbb{R}^{d} \times \mathbb{R}_{>0}$. Fig. \ref{FigDoubleIntegratorReachSet} illustrates this for the double integrator (i.e., $d=2$ case) with $\mathcal{X}_{0}\equiv\{\bm{x}_{0}\}$.

Our motivation behind studying the geometry of the integrator reach sets is twofold. 

\emph{First}, integrators are simple prototypical linear time invariant (LTI) systems that feature prominently in the systems-control literature on reach set computation (see e.g., \cite{maksarov1996state,KurzhanskiValyiBook}, \cite[Ch. 3,4]{varaiyabook},\cite{fainekos2009temporal}). Despite their ubiquity and pedagogical importance, very little is known about the specific geometry of the integrator reach sets. Existing results come in two flavors: rather generic statements (e.g., that these sets are compact and convex), and specific approximation algorithms (e.g., ellipsoidal \cite{KurzhanskiValyiBook,durieu2001multi,varaiyabook} and zonotopic \cite{althoff2007reachability} inner and outer-approximation). To minimize conservatism in numerically approximating the true reach set while preserving safety, one often considers minimum  volume outer-approximation via specific algorithms, see e.g., \cite{halder2018parameterized,halder2018smallest}. In such context, not knowing the volume or diameter of the true reach set hinders a quantitative assessment of whether one algorithm is better than other, in terms of approximating the true reach set. Consequently, one has to content with graphical or statistical assessments about the quality of approximation.      

\emph{Second}, many nonlinear control systems of practical interest, such as aerial and ground vehicles are differentially flat \cite{fliess1995flatness,murray1995,fliess1999lie}, meaning they can be put in a chain of integrator (i.e., Brunovsky canonical) form via a nonlinear change of coordinates. Thus, having an analytical handle on the integrator reach set (in feedback linearized coordinates, see \cite[Thm. 4.1]{nieuwstadt1998differential}) can help compute or approximate the reach set in the original state space. 

In this paper, we present basic convex geometry of the integrator reach sets in finite dimensions with bounded control, and derive closed-form expressions for the volume and the diameter of the same. From the authors' perspective, this paper makes fundamental systems-theoretic (not algorithmic) contribution. Our hope is that the ensuing development will be useful to systems-control researchers using reach set as a construct in applications such as motion planning, and will provide the foundation for development and benchmarking of algorithms.

This paper is structured as follows. Section \ref{SecSptFnCalculus} provides a brief review of the calculus of support functions, and deduces the support function for the integrator reach set. In Section \ref{SubsecVol}, a closed-form formula for the volume of the integrator reach set is derived. Along the way, we establish that the said reach set is a zonoid. In Section \ref{SubsecWidth}, we provide a closed-form formula for the diameter of the integrator reach set. Section \ref{SecConclusions} concludes the paper.

%%%%%%%%%%%%%%%%%%%%%%%%%%%%%%%%%%%%%%%%%%%%%%%%%%%%%%%%%%%%%%%%%%%%%%%%%%%%

\begin{figure*}[t!]
    \centering
    \begin{subfigure}[t]{0.47\textwidth}
        \centering
        \includegraphics[width=\textwidth]{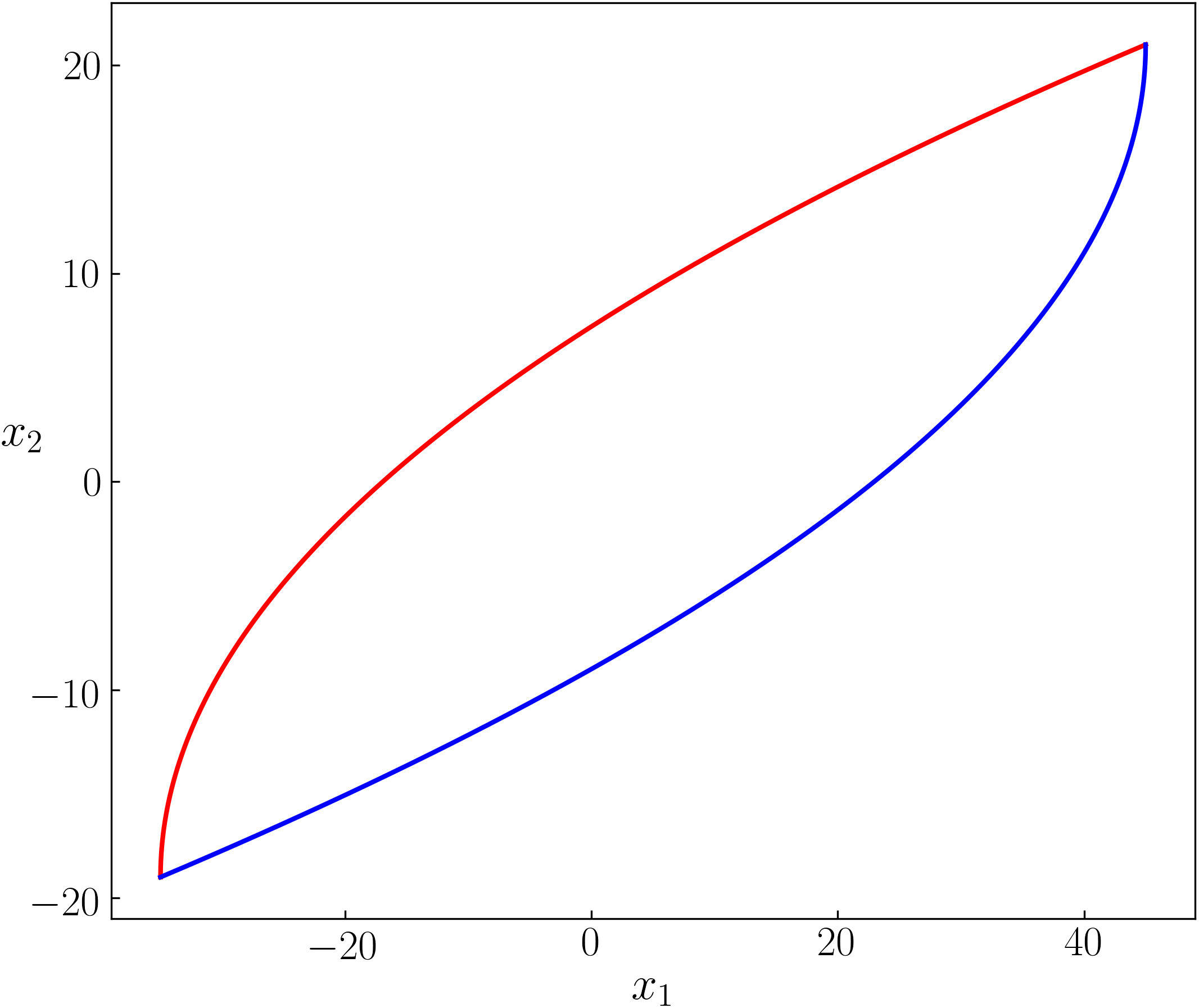}
        \caption{The reach set $\mathcal{R}\left(\{\bm{x}_{0}\},t\right)$ for the double integrator at $t=4$ is the compact convex set enclosed by the upper (\emph{red}) and the lower (\emph{blue}) curves shown.}
    \end{subfigure}%
    ~ 
    \begin{subfigure}[t]{0.5\textwidth}
        \centering
        \includegraphics[width=\textwidth]{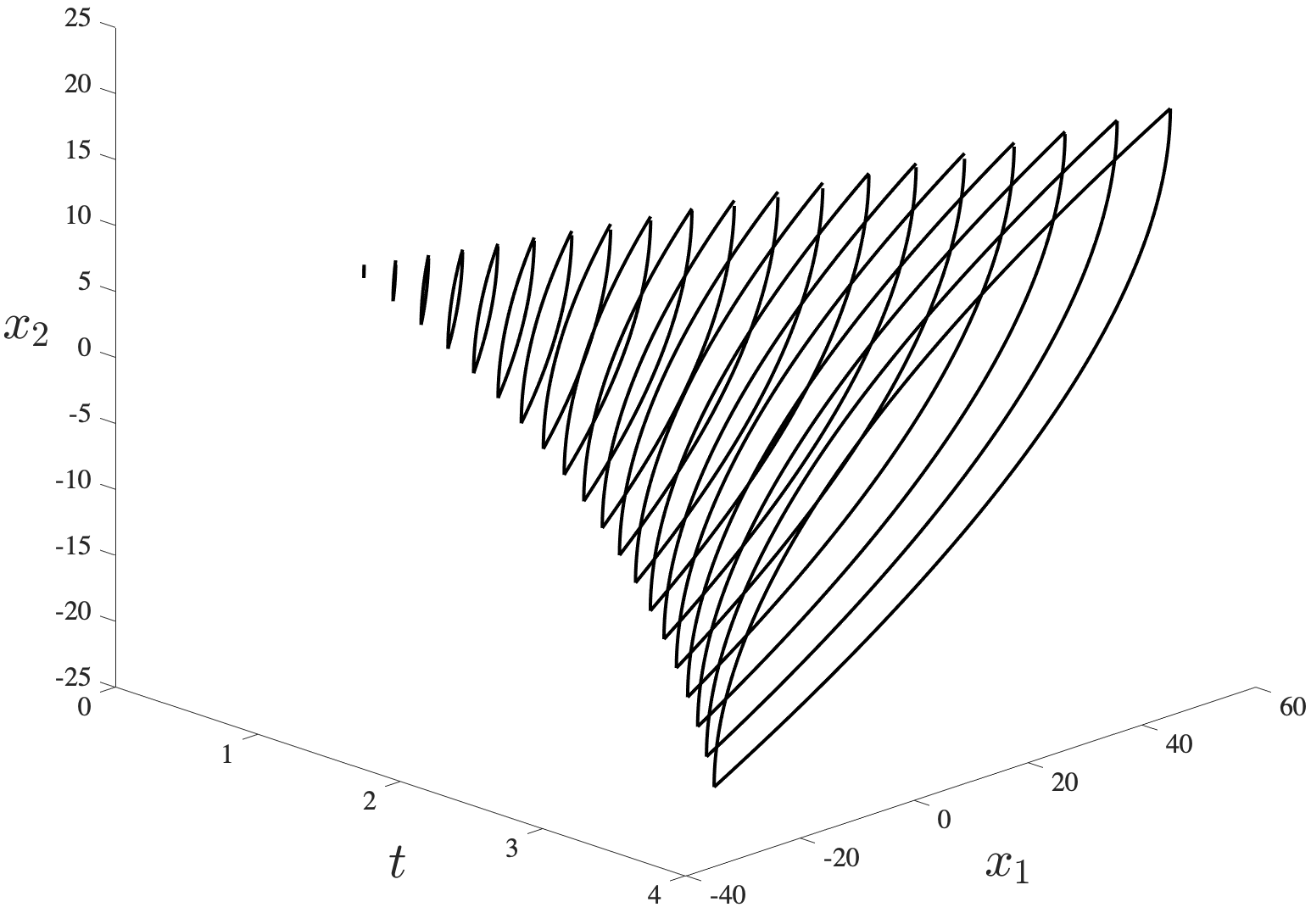}
        \caption{The wireframe plot of the reachable tube $\overline{\mathcal{R}}\left(\{\bm{x}_{0}\},t\right)$ for the double integrator with $t\in[0,4]$, shown here as the union of 20 reach sets. Here, the tube is non-convex even though its time slices (i.e., reach sets) are convex.}
    \end{subfigure}
\caption{\small{(a) The reach set (\ref{DefReachSet}), and (b) the reachable tube (\ref{DefReachableTube}) for the double integrator $(d=2)$ with $\bm{x}_{0} = \left(1,1\right)^{\top}$ and $\mu = 5$.}}
\vspace*{-0.1in}
\label{FigDoubleIntegratorReachSet}
\end{figure*}

%%%%%%%%%%%%%%%%%%%%%%%%%%%%%%%%%%%%%%%%%%%%%%%%%%%%%%%%%%%%%%%%%%%%%%%%%%%%

\section{Support Function Calculus}\label{SecSptFnCalculus}
\subsection{Preliminaries}\label{SubsecSptFnPrelim}
A basic descriptor of a compact convex set $\mathcal{K} \subset \mathbb{R}^{d}$, is its support function $h_{\mathcal{K}}(\cdot)$ given by
\begin{align}
h_{\mathcal{K}}(\bm{y}) := \underset{\bx\in\mathcal{K}}{\sup}\:\{\langle\by,\bx\rangle \mid \bm{y}\in\mathbb{R}^{d}\},
\label{DefSptFn}	
\end{align}
where $\langle\cdot,\cdot\rangle$ denotes the standard Euclidean inner product. Geometrically, $h_{\mathcal{K}}(\bm{y})$ gives the signed distance of the supporting hyperplane of $\mathcal{K}$ with outer normal vector $\bm{y}$, measured from the origin. Since a compact convex set $\mathcal{K}$ is the intersection of its supporting halfspaces, the function $h_{\mathcal{K}}(\cdot)$ characterizes $\mathcal{K}$ by specifying the location of its supporting hyperplanes, parameterized by their outer normal vectors $\bm{y}\in\mathbb{R}^{d}$. A different way to view this characterization is the following \cite[Theorem 13.2]{rockafeller1970convex}: $h_{\mathcal{K}}(\cdot)$ is the Legendre-Fenchel conjugate of the indicator function of the set $\mathcal{K}$.

Several properties of the support function are well-known:
(i) $h_{\mathcal{K}}$ is a convex function of $\bm{y}\in\mathbb{R}^{d}$, and its epigraph is a convex cone.\\
(ii) $h_{\mathcal{K}}$ is positive homogeneous, i.e., $h_{\mathcal{K}}\left(\alpha\by\right)=\alpha h_{\mathcal{K}}\left(\by\right)$ for $\alpha>0$; also, $h_{\mathcal{K}}$ is sub-additive, i.e., $h_{\mathcal{K}}\left(\by + \bm{z}\right) \leq h_{\mathcal{K}}\left(\by\right) + h_{\mathcal{K}}\left(\bm{z}\right)$ for all $\by,\bm{z}\in\mathbb{R}^{d}$.\\
(iii) For $\mathcal{K}_{1},\mathcal{K}_{2}$ compact convex, $\mathcal{K}_{1}\subseteq\mathcal{K}_{2}$ if and only if $h_{\mathcal{K}_{1}}(\by) \leq h_{\mathcal{K}_{2}}(\by)$ for all $\by\in\mathbb{R}^{d}$.\\
(iv) For $\bm{M}\in\mathbb{R}^{d\times d}$ and $\bm{m},\by\in\mathbb{R}^{d}$, $h_{\bm{M}\mathcal{K}+\bm{m}}\left(\by\right) = \langle\by,\bm{m}\rangle + h_{\mathcal{K}}\left(\bm{M}^{\top}\by\right)$.\\
(v) For $\mathcal{K}_{1},\mathcal{K}_{2},\hdots,\mathcal{K}_{r}$ compact convex, and for all $\by\in\mathbb{R}^{d}$,
\[h_{\mathcal{K}_{1}\dotplus\hdots\dotplus\mathcal{K}_{r}}(\by) = h_{\mathcal{K}_{1}}(\by) + \hdots + h_{\mathcal{K}_{r}}(\by),\]
 \[h_{{\rm{ConvexHull}}\left(\mathcal{K}_{1}\cup\hdots\cup\mathcal{K}_{r}\right)}(\by) = \max\{h_{\mathcal{K}_{1}} (\by), \hdots,  h_{\mathcal{K}_{r}}(\by)\},\] 
 \[h_{\mathcal{K}_{1}\cap\hdots\cap\mathcal{K}_{r}}(\by) = \underset{\by_{1} + \hdots + \by_{r} =\by}{\inf}\{h_{\mathcal{K}_{1}}(\by_{1}) + \hdots + h_{\mathcal{K}_{r}}(\by_{r})\}.\]
 
Due to property (ii), a compact convex set in $\mathbb{R}^{d}$ is uniquely determined by its support function restricted to the unit sphere $\mathbb{S}^{d-1}$. We will also need the following result that builds on Lemma \ref{LemmaSetConvergenceSptFnConvergence} in Appendix \ref{AppConvergenceConvexSets}.
\begin{proposition}\label{PropositionSptFnIntegral}
(\textbf{Support function of the integral of a set-valued function}) Let $F(s)$ be a point-to-set function, and denote its support function as $h_{F(s)}\left(\bm{y}\right) \equiv h(s,\bm{y})$ for any $\bm{y}\in\mathbb{R}^{d}$. Then,
\[h_{\int_{0}^{t}F(s)\differential s}\left(\bm{y}\right) = \int_{0}^{t} h\left(s,\bm{y}\right)\differential s.\]	
\end{proposition}

\begin{proof}
For any $\bm{y}\in\mathbb{R}^{d}$, we have
\begin{align*}
h_{\int_{0}^{t}F(s)\differential s}\left(\bm{y}\right) &\stackrel{\text{(\ref{DefIntegralOfSetValuedFn})}}{=} h_{\lim_{\Delta\downarrow 0}\:\sum_{i=0}^{\lfloor\frac{t}{\Delta}\rfloor}\Delta F(i\Delta)}\left(\bm{y}\right)\nonumber\\
&= \underset{\bx\in\lim_{\Delta\downarrow 0}\:\sum_{i=0}^{\lfloor\frac{t}{\Delta}\rfloor}\Delta F(i\Delta)}{\sup}\big\langle\bm{y},\bx\big\rangle\nonumber\\
&\hspace*{-0.2in}\stackrel{\text{(Lemma \ref{LemmaSetConvergenceSptFnConvergence})}}{=} \displaystyle\lim_{\Delta\downarrow 0}\underset{\bx\in\sum_{i=0}^{\lfloor\frac{t}{\Delta}\rfloor}\Delta F(i\Delta)}{\sup}\big\langle\bm{y},\bx\big\rangle\nonumber\\
&= \displaystyle\lim_{\Delta\downarrow 0}h_{\sum_{i=0}^{\lfloor\frac{t}{\Delta}\rfloor}\Delta F(i\Delta)}\left(\bm{y}\right)\nonumber\\
&= \displaystyle\lim_{\Delta\downarrow 0}\displaystyle\sum_{i=0}^{\lfloor\frac{t}{\Delta}\rfloor}\Delta h\left(i\Delta,\bm{y}\right)	\nonumber\\
&= \int_{0}^{t} h\left(s,\bm{y}\right)\differential s,
\end{align*}
wherein the last but one line used the properties (iv)-(v) for the support function.	
\end{proof}
\begin{remark}
A special case of the above result for continuous-time LTI systems was derived in \cite[Proposition 2]{girard2008efficient}. Compared to the same, both the statement and proof of Proposition \ref{PropositionSptFnIntegral} are general (valid for any point-to-set function).
\end{remark} 

In the following, we will derive the support function of the set (\ref{SetValuedIntegral}), and show how geometric quantities of interest can be derived from the same.

\subsection{Support Function of the Integrator Reach Set}\label{SubsecSptFnIntegrator}
From (\ref{SetValuedIntegral}) and the properties (iv)-(v) in Section \ref{SubsecSptFnPrelim}, 
\begin{align}
h_{\mathcal{R}\left(\mathcal{X}_{0},t\right)}\left(\by\right) = h_{\mathcal{X}_{0}}\!\!\left(\exp\left(t\bm{A}^{\!\top}\right)\by\right) \!+\! h_{\int_{0}^{t}\!\exp\left(s\bm{A}\right)\bm{b}\left[-\mu,\mu\right]\differential s}\left(\by\right).
\label{SptFnIntegrator}	
\end{align}
Using Proposition \ref{PropositionSptFnIntegral} and property (iv), we simplify (\ref{SptFnIntegrator}) as
\begin{align}
h_{\mathcal{R}\left(\mathcal{X}_{0},t\right)}\left(\by\right) =& h_{\mathcal{X}_{0}}\left(\exp\left(t\bm{A}^{\top}\right)\by\right) \nonumber\\
&+\int_{0}^{t} h_{\bm{b}\left[-\mu,\mu\right]}\left(\exp\left(s\bm{A}^{\top}\right)\by\right)\differential s.
\label{SptFnIntegratorSimplified}	
\end{align}
Noting the structure of the state transition matrix from Appendix \ref{AppSTM} (equation (\ref{STM})), and that
\[h_{\bm{b}\left[-\mu,\mu\right]}\left(\by\right) = \underset{u\in\left[-\mu,\mu\right]}{\sup}\langle\by,\bm{b}u\rangle = \mu \lvert\langle\by,\bm{b}\rangle\rvert,\] 
we can rewrite (\ref{SptFnIntegratorSimplified}) as 
\begin{align}
h_{\mathcal{R}\left(\mathcal{X}_{0},t\right)}\left(\by\right) =&  
\underset{\bx_{0}\in\mathcal{X}_{0}}{\sup}\langle\by,\exp\left(t\bm{A}\right)\bx_{0}\rangle + \mu\int_{0}^{t}\lvert\langle\by,\bm{\xi}(s)\rangle\rvert\:\differential s,
\label{SptFnIntegratorFinal}	
\end{align}
where $\bm{\xi}(s)$ is the last column of the matrix $\exp(s\bm{A})$, i.e.,
\begin{align}
\bm{\xi}(s) := \begin{pmatrix}\frac{s^{d-1}}{(d-1)!} & \frac{s^{d-2}}{(d-2)!} & \hdots & s & 1\end{pmatrix}^{\top}.
\label{defxi}	
\end{align}
Expressions (\ref{SptFnIntegratorFinal})-(\ref{defxi}) describe the support function of the integrator reach set. 

\section{Functionals of the Integrator Reach Set}
We now show how the ideas presented thus far can be used to compute two functionals of the reach set that are of practical interest, viz. the \emph{volume}, and the \emph{diameter} or \emph{maximal width} of the set. Both these functionals measure the ``size" of the reach set. 

\subsection{Volume}\label{SubsecVol}
We denote the ($d$-dimensional) volume of the integrator reach set in $d$-dimensions, as $\vol\left(\mathcal{R}\left(\mathcal{X}_{0},t\right)\right)$. It can be written in terms of the support function:
\begin{align}
\vol\left(\mathcal{R}\left(\mathcal{X}_{0},t\right)\right) = \dfrac{1}{d}\!\displaystyle\int_{\mathbb{S}^{d-1}}\!\!h_{\mathcal{R}(\mathcal{X}_{0},t)}\left(\bm{\eta}\right)\:\differential S_{\mathcal{R}(\mathcal{X}_{0},t)}\left(\bm{\eta}\right),
\label{volviasptfn}	
\end{align}
where $\bm{\eta}\in\mathbb{S}^{d-1}$ (the Euclidean unit sphere imbedded in $\mathbb{R}^{d}$), and $\differential S_{\mathcal{R}(\mathcal{X}_{0},t)}$ denotes the differential of the surface area measure on $\mathcal{R}(\mathcal{X}_{0},t)$. Let $\differential S$ be the surface area measure on $\mathbb{S}^{d-1}$. Then, we can rewrite (\ref{volviasptfn}) as (see e.g., \cite{lutwak1993the})
\begin{align}
\vol\left(\mathcal{R}\left(\mathcal{X}_{0},t\right)\right) = \dfrac{1}{d}\!\displaystyle\int_{\mathbb{S}^{d-1}}\!\!h_{\mathcal{R}\left(\mathcal{X}_{0},t\right)}\left(\bm{\eta}\right)\dfrac{\differential S_{\mathcal{R}\left(\mathcal{X}_{0},t\right)}}{\differential S}(\bm{\eta})\:\differential S\left(\bm{\eta}\right),
\label{volRadonNikodym}	
\end{align}
where the Radon-Nikodym derivative $\differential S_{\mathcal{R}}/\differential S : \mathbb{S}^{d-1} \mapsto \mathbb{R}$. However, computing (\ref{volRadonNikodym}) with (\ref{SptFnIntegratorFinal}) seems unwieldy as we lack analytical handle on the surface measure of the integrator reach set.

To pursue an alternative strategy for volume computation, we uniformly discretize the interval $[0,t]$ into $n$ subintervals
\[\bigg[\frac{(i-1)t}{n}, \frac{it}{n}\bigg), \quad i=1, 2, \hdots, n,\]
with $(n+1)$ breakpoints $\{t_{i}\}_{i=0}^{n}$, where $t_{i}:=it/n$ for $i=0,1,\hdots,n$. We also consider singleton $\mathcal{X}_{0}\equiv\{\bx_{0}\}$. Then, from (\ref{SetValuedIntegral}), we have
\begin{align*}
\vol\left(\mathcal{R}\left(\{\bx_{0}\},t\right)\right) &= \vol\left(\int_{0}^{t}\!\!\!\exp\left(s\bm{A}\right)\bm{b}\left[-\mu,\mu\right]\differential s\right)\nonumber\\
&= \vol\left(\displaystyle\lim_{n\rightarrow\infty} \displaystyle\sum_{i=0}^{n}\frac{t}{n}\exp\left(t_{i}\bm{A}\right)\bm{b}\left[-\mu,\mu\right]\right),
\end{align*}
where the last line used the definition of the set-valued integral (\ref{DefIntegralOfSetValuedFn}), and that each subinterval is of length $t/n$. We simplify the above expression by taking the scaling factor $\mu t/n$ outside the Minkowski sum, interchanging the limit and $\vol$ (via Lemma \ref{LemmaSetConvergenceVolConvergence}), and using the homogeneity of the $\vol(\cdot)$ operator, to obtain
\begin{align}
\vol\left(\mathcal{R}\left(\{\bx_{0}\},t\right)\right) = \!\displaystyle\lim_{n\rightarrow\infty}\!\left(\dfrac{\mu t}{n}\right)^{\!\!d}\!\!\vol\left(\displaystyle\sum_{i=0}^{n}\exp\left(t_{i}\bm{A}\right)\bm{b}\left[-1,1\right]\right).
\label{vol}	
\end{align}
At this point, we notice that the set
\begin{align}
\displaystyle\sum_{i=0}^{n}\exp\left(t_{i}\bm{A}\right)\bm{b}\left[-1,1\right]
\label{DefSet}	
\end{align}
is a Minkowski sum of $n+1$ intervals, each interval being rotated and scaled in $\mathbb{R}^{d}$ via different linear transformations $\exp(t_{i}\bm{A})$, $i=0,1,\hdots,n$. To proceed further, we recall few facts about zonotopes \cite{bolker1969class, mcmullen1971on,shephard1974combinatorial}.

\subsubsection{Zonotopes and zonoids}
A $d$-dimensional \emph{zonotope} $\mathcal{Z}_{n}$ is defined as a Minkowski sum of $n$ line segments:
\begin{align}
\!\!\!\!\mathcal{Z}_{n} \!:= \!\bigg\{\!\sum_{j=1}^{n}\!\gamma_{j}\bm{v}_{j} \!\mid\! \gamma_{j}\in[-1,1], \bm{v}_{j}\in\mathbb{R}^{d}, j=1,\hdots,n\bigg\},
\label{DefZonotope}	
\end{align}
where the vectors $\{\bm{v}_{j}\}_{j=1}^{n}$ are called the ``generators" of $\mathcal{Z}_{n}$. To make this explicit, one often writes 
\[\mathcal{Z}_{n}\equiv\mathcal{Z}_{n}\left(\{\bm{v}_{j}\}_{j=1}^{n}\right) \subset \mathbb{R}^{d}.\] 
Equivalently, (\ref{DefZonotope}) can be seen as an affine transformation of the unit cube in $\mathbb{R}^{d}$. From Section \ref{SubsecSptFnPrelim}, the support function of (\ref{DefZonotope}) is 
\begin{align}
h_{\mathcal{Z}_{n}}(\bm{y}) = \displaystyle\sum_{j=1}^{n}\lvert\langle\bm{y},\bm{v}_{j}\rangle\rvert, \quad \bm{y}\in\mathbb{R}^{d}.
\label{sptfnzonotope}	
\end{align}
Conversely, a set $\mathcal{Z}_{n}$ with support function of the form (\ref{sptfnzonotope}) must be a zonotope \cite[p. 297]{schneider1983zonoids}. See \cite{witsenhausen1968} for a support function inequality characterization of zonotopes. The limiting (where the limit is w.r.t. the Hausdorff metric, see Appendix \ref{AppConvergenceConvexSets}) compact convex set $\mathcal{Z}_{\infty}:= \lim_{n\rightarrow\infty}\mathcal{Z}_{n}$ is termed as the ``zonoid" \cite{bolker1969class,schneider1983zonoids}.

The following formula for the $d$-dimensional volume of (\ref{DefZonotope}) appears in \cite[eqn. (57)]{shephard1974combinatorial}, who attributes it to \cite{mcmullen1971on}:
\begin{align}
\vol\left(\mathcal{Z}_{n}\right) = 2^{d} \!\!\!\!\displaystyle\sum_{1 \leq j_{1} < j_{2} < \hdots < j_{d} \leq n} \!\!\lvert \det\left(\bm{v}_{j_{1}} \vert \bm{v}_{j_{2}} \vert \hdots \vert \bm{v}_{j_{d}}\right) \rvert,
\label{VolZonotope}	
\end{align}
where the summands are (non-negative) determinants of the $d\times d$ matrices, as shown. The formula (\ref{VolZonotope}) also appears in \cite[exercise 7.19]{ziegler2012lectures}, and can be derived by decomposing $\mathcal{Z}_{n}$ into parallelopipeds \cite[Fig. 5]{shephard1974combinatorial}, whose volumes are given by the summand determinants.

In our context, the aforesaid facts have two immediate consequences, summarized below.
\begin{proposition}\label{PropositionZonotopeZonoid}
(i) Recall that 
\[\bm{\xi}\left(t_{i}\right) = \bm{\xi}\left(it/n\right) \in\mathbb{R}^{d},\quad i=0,1,\hdots,n,\] 
where $\bm{\xi}(\cdot)$ is given by (\ref{defxi}). The set (\ref{DefSet}) is a zonotope $\mathcal{Z}_{n+1}\left(\{\bm{\xi}(t_{i})\}_{i=0}^{n}\right)$	 with support function
\[h_{\mathcal{Z}_{n+1}\left(\{\bm{\xi}(t_{i})\}_{i=0}^{n}\right)}(\bm{y}) = \displaystyle\sum_{i=0}^{n}\lvert\langle\bm{y},\bm{\xi}(t_{i})\rangle\rvert, \quad \bm{y}\in\mathbb{R}^{d}.\]
(ii) The limiting compact convex set $\int_{0}^{t}\exp(s\bm{A})\bm{b}[-1,1]\differential s$ is a zonoid.
\end{proposition}

\subsubsection{Back to volume computation}
Thanks to Proposition \ref{PropositionZonotopeZonoid}, we can use (\ref{VolZonotope}) to rewrite the right-hand-side (RHS) of (\ref{vol}) as
\begin{align}
\!\!\!\!\left(2\mu t\right)^{d}\displaystyle\lim_{n\rightarrow\infty}\dfrac{1}{n^{d}} \!\displaystyle\sum_{0 \leq i_{1} < i_{2} < \hdots < i_{d} \leq n} \!\!\lvert \det\left(\bm{\xi}(t_{i_{1}}) \vert \bm{\xi}(t_{i_{2}}) \vert \hdots \vert \bm{\xi}(t_{i_{d}})\right) \rvert.
\label{SimplifyVolRHS}	
\end{align}
This leads to the following result (proof in Appendix \ref{AppProofVolThm}).
\begin{theorem}\label{ThmVolIntegratorReachSet} 
(\textbf{Volume of the integrator reach set})
Given controlled dynamics (\ref{IntegratorDyn})-(\ref{StateMatrices}), and a fixed $\bm{x}_{0}\in\mathbb{R}^{d}$, let $\mathcal{X}_{0} \equiv \{\bm{x}_{0}\}$. At time $t>0$, the volume of the reach set (\ref{DefReachSet}) is
\begin{align}
\vol\left(\mathcal{R}\left(\{\bx_{0}\},t\right)\right) = \dfrac{(2\mu)^{d}t^{d(d+1)/2}}{\displaystyle\prod_{k=1}^{d-1}k!}\displaystyle\lim_{n\rightarrow\infty}\displaystyle\frac{1}{n^{d(d+1)/2}}\times\nonumber\\
\displaystyle\sum_{0 \leq i_{1} < i_{2} < \hdots < i_{d} \leq n} \quad\displaystyle\prod_{1 \leq \alpha < \beta \leq d} \left(i_{\beta} - i_{\alpha}\right).
\label{FinalVolFormula}	
\end{align}	
\end{theorem}

\begin{remark}
Since the sum
\[\displaystyle\sum_{0 \leq i_{1} < i_{2} < \hdots < i_{d} \leq n} \quad\displaystyle\prod_{1 \leq \alpha < \beta \leq d} \left(i_{\beta} - i_{\alpha}\right)\]
must return a polynomial in $n$ of degree $d(d+1)/2$, hence the limit in (\ref{FinalVolFormula}) will simply extract the leading coefficient of this polynomial, i.e., the coefficient of $n^{d(d+1)/2}$. A corollary is that the said limit (which is a function of $d$) is well defined.	
\end{remark}

\begin{remark}
If the set of initial conditions $\mathcal{X}_{0}$ is not singleton, then computing the volume of the reach set requires computing volume of a Minkowski sum: $\vol\left(\mathcal{R}\left(\mathcal{X}_{0},t\right)\right) = \vol\left(\mathcal{R}_{0} \dotplus \mathcal{R}_{t}\right)$, where $\mathcal{R}_{0}:=\exp(t\bm{A})\mathcal{X}_{0}$, and $\mathcal{R}_{t} := \int_{0}^{t}\exp\left(s\bm{A}\right)\bm{b}\left[-\mu,\mu\right]\differential s$. 
\end{remark}

To illustrate the use of (\ref{FinalVolFormula}), consider $d=2$, i.e., the case of a double integrator for which the shape of the reach set is as in Fig. \ref{FigDoubleIntegratorReachSet}(a). From (\ref{FinalVolFormula}), its area is
\begin{align}
4\mu^{2}t^{3}\displaystyle\lim_{n\rightarrow\infty}\displaystyle\frac{1}{n^{3}}\displaystyle\sum_{i=0}^{n}\sum_{j=i+1}^{n}(j-i), 
\label{VolDoubleIntegrator}	
\end{align}
wherein we renamed the indices $(i_{1},i_{2})\mapsto(i,j)$. Straightforward calculation gives
\[\displaystyle\sum_{i=0}^{n}\sum_{j=i+1}^{n}(j-i) = \frac{1}{6}n\left(n+1\right)\left(n+2\right);\] 
thus the limit in (\ref{VolDoubleIntegrator}) evaluates to $1/6$, and hence the area equals $2\mu^{2}t^{3}/3$. We note that for the double integrator, it is possible to derive the area formula $2\mu^{2}t^{3}/3$ by first deriving the equations of the bounding curves \cite[p. 111]{varaiyabook} (\emph{red and blue} curves in Fig. \ref{FigDoubleIntegratorReachSet}(a)), and then computing the area enclosed by the two \cite[Appendix A]{gianfortone2018thesis}. However, it is difficult to generalize this approach to higher dimensions.  

As another example, consider $d=3$ (the triple integrator). Then (\ref{FinalVolFormula}) reduces to
\begin{align}
4\mu^{3}t^{6}\displaystyle\lim_{n\rightarrow\infty}\displaystyle\frac{1}{n^{6}}\displaystyle\sum_{i=0}^{n}\sum_{j=i+1}^{n}\sum_{k=j+1}^{n}(k-j)(k-i)(j-i), 
\label{VolTripleIntegrator}	
\end{align}
where again we renamed the indices $(i_{1},i_{2},i_{3})\mapsto(i,j,k)$. Noting that 
\begin{align*}
\displaystyle\sum_{i=0}^{n}\displaystyle\sum_{j=i+1}^{n}\displaystyle\sum_{k=j+1}^{n}(k-j)(j-i)(k-i) \\
= \frac{1}{180}(n-1)n(n+1)^{2}(n+2)(n+3),	
\end{align*}
we find that the limit in (\ref{VolTripleIntegrator}) evaluates to $1/180$, and therefore, the volume of the triple integrator reach set is $\mu^{3}t^{6}/45$.

Similar calculation shows that the (4-dimensional) volume of the quadruple integrator (i.e., $d=4$) is $\mu^{4}t^{10}/18900$.

\subsection{Width}\label{SubsecWidth}
The \emph{width} \cite[p. 42]{schneider2014book} of the reach set $\mathcal{R}\left(\mathcal{X}_{0},t\right)$ is 
\begin{align}
w_{\mathcal{R}\left(\mathcal{X}_{0},t\right)}(\bm{\eta}) := h_{\mathcal{R}(\mathcal{X}_{0},t)}\left(\bm{\eta}\right) + h_{\mathcal{R}(\mathcal{X}_{0},t)}\left(-\bm{\eta}\right),
\label{DefWidth}	
\end{align}
where $\bm{\eta}\in\mathbb{S}^{d-1}$, and $h_{\mathcal{R}(\mathcal{X}_{0},t)}\left(\cdot\right)$ is given by (\ref{SptFnIntegratorFinal}). That is, (\ref{DefWidth}) gives the width of the reach set in the direction $\bm{\eta}$. 

For $\mathcal{X}_{0} \equiv \{\bm{x}_{0}\}$ (a singleton set), combining (\ref{SptFnIntegratorFinal}) and (\ref{DefWidth}), we have
\begin{align}
w_{\mathcal{R}\left(\{\bm{x}_{0}\},t\right)}(\bm{\eta}) &= \mu\int_{0}^{t} \bigg\{\lvert\langle\bm{\eta},\bm{\xi}(s)\rangle\rvert + \lvert\langle-\bm{\eta},\bm{\xi}(s)\rangle\rvert\bigg\}\:\differential s \nonumber\\
	&= 2\mu\int_{0}^{t} \lvert\langle\bm{\eta},\bm{\xi}(s)\rangle\rvert\:\differential s,\label{widthformula}
\end{align}
where the last equality follows from the fact that $\bm{\xi}(s)$ in (\ref{defxi}) is component-wise non-negative for all $0 \leq s \leq t$. Fig. \ref{FigWidthReachSet} shows how the width of the integrator reach set for $d=2$ varies over $\theta\in\mathbb{S}^{1}$ (i.e., $\bm{\eta}\equiv (\cos\theta, \sin\theta)^{\top}$ in this case).  

\begin{figure}[t]
        \centering
        \includegraphics[width=\linewidth]{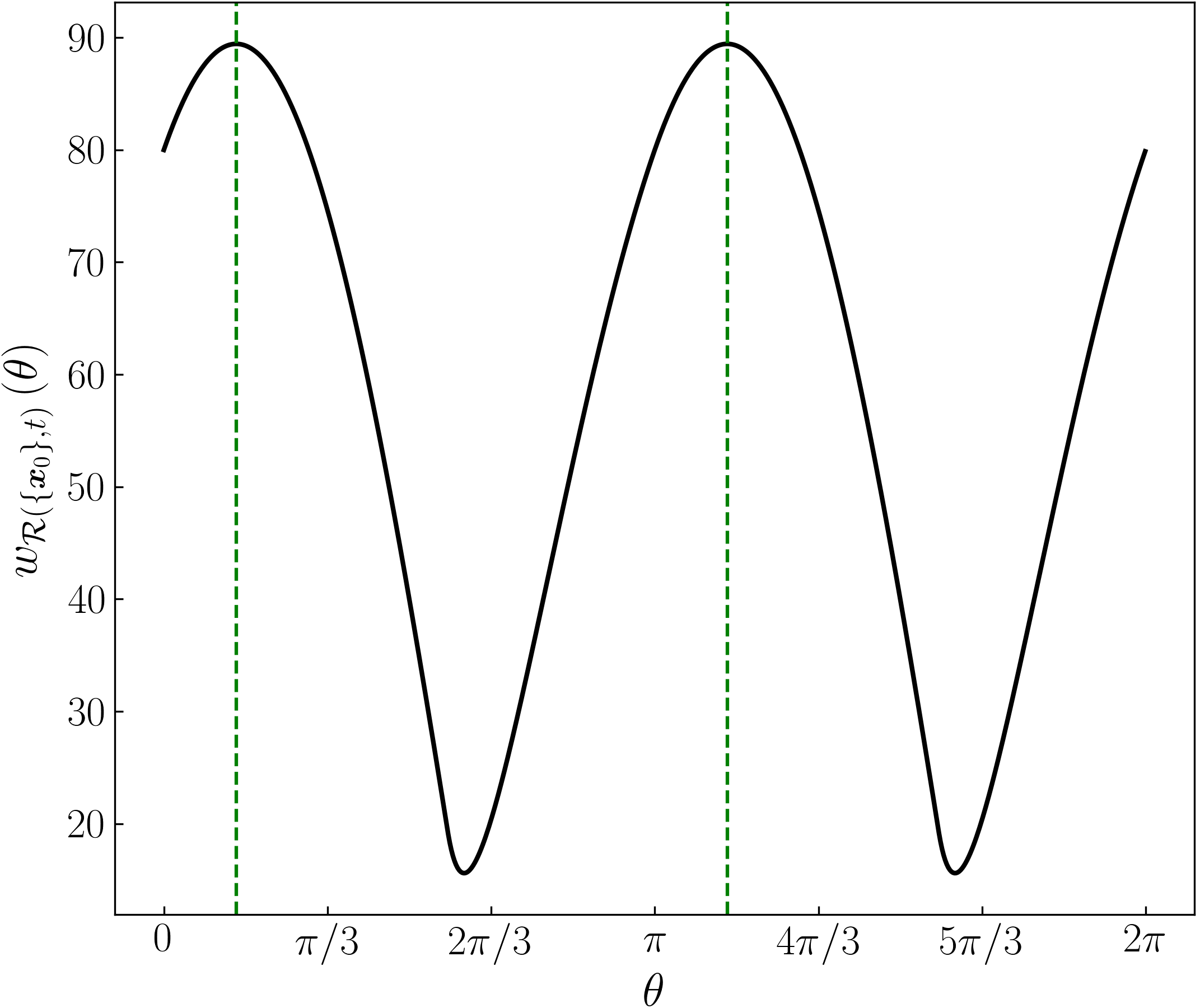}
        \caption{The solid line shows the width of the double integrator reach set $\mathcal{R}\left(\{\bm{x}_{0}\},t\right)$ shown in Fig. \ref{FigDoubleIntegratorReachSet}(a), as function of $\theta\in\mathbb{S}^{1}$ given by (\ref{widthformula}). As in Fig. \ref{FigDoubleIntegratorReachSet}(a), here $\bm{x}_{0} = \left(1,1\right)^{\top}$, $\mu = 5$, $t=4$. The dashed vertical lines show the location of the maximizers $\theta_{r}^{\max} = r\pi + \arctan(2/t)$, $r=0,1$, for the width, i.e., the directions defining the diameter of the reach set (see Theorem \ref{ThmDiamIntegratorReachSet} and the discussion thereafter).}
\vspace*{-0.1in}
\label{FigWidthReachSet}
\end{figure}

The \emph{diameter} of the reach set is the maximal width:
\begin{align}
{\rm{diam}}\left(\mathcal{R}\left(\mathcal{X}_{0},t\right)\right) := \underset{\bm{\eta}\in\mathbb{S}^{d-1}}{\max}\:w_{\mathcal{R}\left(\mathcal{X}_{0},t\right)}(\bm{\eta}).
\label{DefDiam}	
\end{align}
It is clear from the support function property (ii) in Section \ref{SubsecSptFnPrelim} that both the width and the diameter are Minkowski sub-additive, i.e., $w_{\mathcal{K}_{1}\dotplus\mathcal{K}_{2}}(\cdot) \leq w_{\mathcal{K}_{1}}(\cdot) + w_{\mathcal{K}_{2}}(\cdot)$, and ${\rm{diam}}\left(\mathcal{K}_{1}\dotplus\mathcal{K}_{2}\right) \leq {\rm{diam}}\left(\mathcal{K}_{1}\right) + {\rm{diam}}\left(\mathcal{K}_{2}\right)$ for compact convex $\mathcal{K}_{1},\mathcal{K}_{2}$.

Notice that the integrand in (\ref{widthformula}) is a convex function in $\bm{\eta}$ for each $s\in[0,t]$, and therefore $w_{\mathcal{R}\left(\{\bm{x}_{0}\},t\right)}(\bm{\eta})$ is convex in $\bm{\eta}$; see \cite[p. 79]{boydbook}. So, computing (\ref{DefDiam}) is a concave problem:
\begin{align}
2\mu\:\underset{\bm{\eta}^{\top}\bm{\eta}=1}{\max}\:\displaystyle\int_{0}^{t}\lvert\langle\bm{\eta},\bm{\xi}(s)\rangle\rvert\:\differential s.
\label{ConcaveProblemDiameterOfReachSet}	
\end{align}
We have the following result (proof in Appendix \ref{AppProofDiamThm}).
\begin{theorem}\label{ThmDiamIntegratorReachSet}
(\textbf{Diameter of the integrator reach set})
Given controlled dynamics (\ref{IntegratorDyn})-(\ref{StateMatrices}), and a fixed $\bm{x}_{0}\in\mathbb{R}^{d}$, let $\mathcal{X}_{0} \equiv \{\bm{x}_{0}\}$, and the vector $\bm{\zeta}(t) := \int_{0}^{t}\bm{\xi}(s)\differential s$. At time $t>0$, the diameter (i.e., maximal width) of the reach set (\ref{DefReachSet}) is
\begin{subequations}
\begin{align}
{\rm{diam}}\left(\mathcal{R}\left(\{\bx_{0}\},t\right)\right) &= 2\mu \parallel \bm{\zeta}(t) \parallel_{2} \label{compactformDiamFormula}\\
	&= 2\mu \bigg\{\!\displaystyle\sum_{j=1}^{d}\left(\dfrac{t^{j}}{j!}\right)^{\!\!2}\!\bigg\}^{\!\!\frac{1}{2}}.\label{expandedformDiamFormula}
\end{align}
\label{FinalDiamFormula}		
\end{subequations}		
\end{theorem}
To illustrate Theorem \ref{ThmDiamIntegratorReachSet}, consider again the double integrator for which $\bm{\eta}\equiv (\cos\theta, \sin\theta)^{\top}$, $\theta\in\mathbb{S}^{1}$, and (\ref{ConcaveProblemDiameterOfReachSet}) becomes
\[2\mu\:\underset{\theta\in[0,2\pi)}{\max}\:\displaystyle\int_{0}^{t}\lvert s\cos\theta + \sin\theta\rvert\:\differential s.\]
In this case, $\bm{\zeta}(t) = (t^{2}/2, t)^{\top}$, and from (\ref{maximizereta}) in Appendix \ref{AppProofDiamThm}, we obtain the maximizers
\begin{align*}
\!\begin{pmatrix}
\!\cos\theta^{\max}\!\\
\!\sin\theta^{\max}\!	
\end{pmatrix}
\!= \!\begin{pmatrix}
 \!\dfrac{t^2/2}{\sqrt{t^{4}/4 + t^{2}}}\!\\
  \!\dfrac{t}{\sqrt{t^{4}/4 + t^{2}}}\!	
 \end{pmatrix} \Rightarrow \theta_{r}^{\max} = r\pi + \arctan\left(2/t\right),
%\label{maximizingtheta}	
\end{align*}
for $r=0,1$ (so as to make $\theta\in[0,2\pi)$). In Fig. \ref{FigWidthReachSet}, the location of these maximizers in the horizontal axis are depicted via dashed vertical lines. From (\ref{FinalDiamFormula}), the diameter of the double integrator reach set becomes $\mu t \sqrt{t^{2} + 4}$, which for our choice of parameters in Fig. \ref{FigWidthReachSet}, equals $89.4427$.

Similarly, the diameter of the triple integrator ($d=3$) reach set, from (\ref{expandedformDiamFormula}), equals $(\mu t/3)\sqrt{t^{4} + 9 t^{2} + 36}$.

%%%%%%%%%%%%%%%%%%%%%%%%%%%%%%%%%%%%%%%%%%%%%%%%%%%%%%%%%%%%%%%%%%%%%%%%%%%%%%%%

\section{Conclusions}\label{SecConclusions}
In this paper, we took a close look at the convex geometry of the integrator reach sets. Using support function calculus, we established that the integrator reach set is a zonoid -- a limiting compact convex set of a sequence of zonotopes (Minkowski sum of line segments). This limit is defined in the (two-sided) Hausdorff metric. These ideas enabled us to derive closed-form formula for the volume and diameter of the integrator reach sets. Several examples are given to elucidate the results.

%%%%%%%%%%%%%%%%%%%%%%%%%%%%%%%%%%%%%%%%%%%%%%%%%%%%%%%%%%%%%%%%%%%%%%%%%%%%%%%%

\section{Acknowledgments}
This research was partially supported by Chancellor's Fellowship from the University of California, Santa Cruz.

%%%%%%%%%%%%%%%%%%%%%%%%%%%%%%%%%%%%%%%%%%%%%%%%%%%%%%%%%%%%%%%%%%%%%%%%%%%%%%%%

\appendix{}

\subsection{Integrator State Transition Matrix}\label{AppSTM}
The state matrix $\bm{A}$ for (\ref{IntegratorDyn}) is the upper shift matrix, and hence the state transition matrix $\exp(t\bm{A})$ is upper triangular with all ones in the diagonal, as given below.
\begin{lemma}\label{LemmaSTM}
For the $d\times d$ matrix $\bm{A}$ as in (\ref{StateMatrices}), we have
\begin{align}
\exp(t\bm{A})_{i,j} = \begin{cases}
 	\displaystyle\frac{t^{j-i}}{(j-i)!} & \text{for}\quad i<j,\\
 	1 & \text{for}\quad i=j,\\
 	0 & \text{for}\quad i>j,\\
 \end{cases}
\label{STM}	
\end{align}
where the indices $i,j=1,\hdots,d$. 	
\end{lemma}
\begin{proof}
Since $\bm{A}$ is nilpotent with order $d$ (i.e., $\bm{A}^{d}$ is zero matrix), hence $\exp(t\bm{A}) = \displaystyle\sum_{j=0}^{d-1}\dfrac{\left(t\bm{A}\right)^{j}}{j!}$, from which the result follows.	
\end{proof}

\subsection{Convergence of Compact Convex Sets}\label{AppConvergenceConvexSets}
Let $\{\mathcal{K}_{i}\}_{i\in\mathbb{N}}$ be a sequence of compact convex sets in $\mathbb{R}^{n}$. Define the two-sided Hausdorff distance $\delta_{{\rm{H}}}$ (which is a metric) between any two non-empty sets $\mathcal{P},\mathcal{Q}\subseteq\mathbb{R}^{d}$ as
\begin{align}
\delta_{{\rm{H}}}\left(\mathcal{P},\mathcal{Q}\right) := \max\bigg\{\underset{\bm{p}\in\mathcal{P}}{\sup}\:\underset{\bm{q}\in\mathcal{Q}}{\vphantom{\sup}\inf} \parallel \bm{p} - \bm{q} \parallel_{2},\nonumber\\ 
\underset{\bm{q}\in\mathcal{Q}}{\sup}\:\underset{\bm{p}\in\mathcal{P}}{\vphantom{\sup}\inf} \parallel \bm{p} - \bm{q} \parallel_{2}\bigg\}.
\label{DefHausdorff}	
\end{align}
For $\mathcal{P,Q}$ compact and convex, (\ref{DefHausdorff}) becomes
\begin{align}
\delta_{{\rm{H}}}\left(\mathcal{P},\mathcal{Q}\right) = \underset{\bm{\eta}\in\mathbb{S}^{d-1}}{\sup}\:\lvert h_{\mathcal{P}}(\bm{\eta}) - h_{\mathcal{Q}}(\bm{\eta})\rvert,
\label{HausdorffviaSptFn}	
\end{align}
where $\mathbb{S}^{d-1}$ denotes the Euclidean unit sphere imbedded in $\mathbb{R}^{d}$. The sequence $\{\mathcal{K}_{i}\}_{i\in\mathbb{N}}$ converges to a compact convex set $\mathcal{K}$ (loosely denoted as $\mathcal{K}_{i} \rightarrow \mathcal{K}$) iff $\delta_{{\rm{H}}}\left(\mathcal{K}_{i},\mathcal{K}\right) \rightarrow 0$ as $i\rightarrow\infty$. In fact, the collection of all compact convex sets equipped with the metric $\delta_{{\rm{H}}}$ constitutes a complete\footnote{i.e., every Cauchy sequence of sets is convergent in $\delta_{{\rm{H}}}$, and vice versa.} metric space. The following Lemma will be useful in proving Proposition \ref{PropositionSptFnIntegral} in Section \ref{SubsecSptFnPrelim}.
\begin{lemma}\label{LemmaSetConvergenceSptFnConvergence}
Let $\{\mathcal{K}_{i}\}_{i\in\mathbb{N}}$ be a sequence of compact convex sets in $\mathbb{R}^{n}$. Then, $\mathcal{K}_{i} \rightarrow \mathcal{K} \: \Leftrightarrow h_{\mathcal{K}_{i}}(\cdot) \rightarrow h_{\mathcal{K}}(\cdot)$.	
\end{lemma}
\begin{proof}
We know that $\mathcal{K}_{i} \rightarrow \mathcal{K}$ iff $\delta_{{\rm{H}}}\left(\mathcal{K}_{i},\mathcal{K}\right)\rightarrow 0$. From (\ref{HausdorffviaSptFn}), the latter is equivalent to $h_{\mathcal{K}_{i}}(\cdot) \rightarrow h_{\mathcal{K}}(\cdot)$.	
\end{proof}
We will also need the following.
\begin{lemma}\label{LemmaSetConvergenceVolConvergence}
Let $\{\mathcal{K}_{i}\}_{i\in\mathbb{N}}$ be a sequence of compact convex sets in $\mathbb{R}^{d}$. Let $\vol(\cdot)$ denote the $d$-dimensional volume. If $\mathcal{K}_{i} \rightarrow \mathcal{K}$, then $\vol\left(\mathcal{K}_{i}\right) \rightarrow \vol\left(\mathcal{K}\right)$ as $i\rightarrow\infty$.	
\end{lemma}
\begin{proof}
Follows from continuity of the volume functional \cite[p. 55]{schneider2014book}, and uniform convergence of the support functions of corresponding sets.
\end{proof}

\subsection{Proof of Theorem \ref{ThmVolIntegratorReachSet}}\label{AppProofVolThm}
\begin{proof}[\unskip\nopunct]
For a given $d$-tuple $\{i_{1}, i_{2}, \hdots, i_{d}\}$, where $0 \leq i_{1} < i_{2} < \hdots, < i_{d} \leq n$, let 
\[\Delta(i_{1},i_{2},\hdots,i_{d}) := \det\left(\bm{\xi}(t_{i_{1}}) \vert \bm{\xi}(t_{i_{2}}) \vert \hdots \vert \bm{\xi}(t_{i_{d}}\right))\] 
denote the corresponding summand determinant in (\ref{SimplifyVolRHS}). In the matrix list notation, let us use the vertical bars $\lvert \cdot \rvert$ to mean \emph{the absolute value of the determinant}. From (\ref{defxi}), $\Delta(i_{1},i_{2},\hdots,i_{d})$ equals
\begin{align}
&\begin{vmatrix}
 \dfrac{\left(i_{1}t/n\right)^{d-1}}{(d-1)!} & \dfrac{\left(i_{2}t/n\right)^{d-1}}{(d-1)!} & \hdots & \dfrac{\left(i_{d}t/n\right)^{d-1}}{(d-1)!}\\
  & & &\\
  \dfrac{\left(i_{1}t/n\right)^{d-2}}{(d-2)!} & \dfrac{\left(i_{2}t/n\right)^{d-2}}{(d-2)!} & \hdots & \dfrac{\left(i_{d}t/n\right)^{d-2}}{(d-2)!}\\ 
   &&&\\ 
 \vdots & \vdots & \vdots & \vdots\\
  &&&\\
 i_{1}t/n & i_{2}t/n & \hdots & i_{d}t/n\\
 &&&\\
 1 & 1 & \hdots & 1
 \end{vmatrix}\nonumber\\
&= \dfrac{(t/n)^{1 + 2 + \hdots + (d-1)}}{1!\times2!\times\hdots\times(d-1)! }\begin{vmatrix}
i_{1}^{d-1} & i_{2}^{d-1} & \hdots & i_{d}^{d-1}\\
  & & &\\
 i_{1}^{d-2} & i_{2}^{d-2} & \hdots & i_{d}^{d-2}\\ 
   &&&\\ 
 \vdots & \vdots & \vdots & \vdots\\
  &&&\\
 i_{1} & i_{2} & \hdots & i_{d}\\
 &&&\\
 1 & 1 & \hdots & 1
 \end{vmatrix}, \nonumber\\
&=  \dfrac{(t/n)^{d(d-1)/2}}{\displaystyle\prod_{k=1}^{d-1}k!}\begin{vmatrix}
 1 & 1 & \hdots & 1\\
  & & &\\
   i_{1} & i_{2} & \hdots & i_{d}\\ 
   &&&\\ 
 \vdots & \vdots & \vdots & \vdots\\
  &&&\\  
 i_{1}^{d-2} & i_{2}^{d-2} & \hdots & i_{d}^{d-2}\\
 &&&\\
i_{1}^{d-1} & i_{2}^{d-1} & \hdots & i_{d}^{d-1}\\
 \end{vmatrix},
 \label{DetStructure}
\end{align}
wherein the last but one step brought the multiple for each row outside the determinant, and the last step rearranged the rows (without affecting sign since we are dealing with the absolute value of the determinant). 

Next, notice that the determinant in (\ref{DetStructure}) is the well-known \emph{Vandermonde determinant} that equals \cite[p. 37]{hornjohnson2012}
\begin{align}
\displaystyle\prod_{1 \leq \alpha < \beta \leq d} \left(i_{\beta} - i_{\alpha}\right).
\label{VandermondeDet}	
\end{align}
Combining (\ref{SimplifyVolRHS}), (\ref{DetStructure}) and (\ref{VandermondeDet}), we arrive at (\ref{FinalVolFormula}).
\end{proof}

\subsection{Proof of Theorem \ref{ThmDiamIntegratorReachSet}}\label{AppProofDiamThm}
\begin{proof}[\unskip\nopunct]
Let $\bm{\xi}_{i}(s)$ denote the $i$-th component of the vector $\bm{\xi}(s)$ given by (\ref{defxi}), where $i=1,2,\hdots,d$. We notice that $\lvert\sum_{i=1}^{d}\bm{\eta}_{i}\bm{\xi}_{i}(s)\rvert \leq \sum_{i=1}^{d}\lvert\bm{\eta}_{i}\bm{\xi}_{i}(s)\rvert = \sum_{i=1}^{d}\lvert\bm{\eta}_{i}\rvert\bm{\xi}_{i}(s)$ (since the vector $\bm{\xi}(s)$ is elementwise nonnegative) for any $s\in[0,t]$. Therefore, 	
\begin{align}
\displaystyle\int_{0}^{t}\!\bigg\lvert\displaystyle\sum_{i=1}^{d}\bm{\eta}_{i}\bm{\xi}_{i}(s)\bigg\rvert\:\differential s \leq \displaystyle\sum_{i=1}^{d} \lvert\bm{\eta}_{i}\rvert \!\displaystyle\int_{0}^{t}\!\bm{\xi}_{i}(s)\:\differential s = \displaystyle\sum_{i=1}^{d}\bm{\zeta}_{i}(t)\lvert\bm{\eta}_{i}\rvert.
\label{DiamProofIneq}	
\end{align}
By standard Lagrange multiplier argument, maximizing the weighted $\ell_{1}$-norm $\displaystyle\sum_{i=1}^{d}\bm{\zeta}_{i}(t)\lvert\bm{\eta}_{i}\rvert$ appearing in the RHS of (\ref{DiamProofIneq}) w.r.t. $\bm{\eta}$, subject to $\bm{\eta}^{\top}\bm{\eta}=1$, yields the maximizer
\begin{align}
\bm{\eta}^{\max} = \pm \dfrac{\bm{\zeta}(t)}{\parallel\bm{\zeta}(t)\parallel_{2}},
\label{maximizereta}	
\end{align}
i.e., the unit vector associated with $\bm{\zeta}(t)$ upto plus-minus sign permutations among its components.

Substituting $\bm{\eta}^{\max}$ from (\ref{maximizereta}) back in (\ref{DiamProofIneq}) yields the maximal value attainable by the integral in (\ref{ConcaveProblemDiameterOfReachSet}) as $\parallel\bm{\zeta}(t)\parallel_{2}$. We thus obtain
\[{\rm{diam}}\left(\mathcal{R}\left(\{\bx_{0}\},t\right)\right) = 2\mu \parallel \bm{\zeta}(t) \parallel_{2},\]
which is indeed (\ref{compactformDiamFormula}). The expression (\ref{expandedformDiamFormula}) results from recalling that the components of $\bm{\zeta}(t)$, by definition, are
\[\bm{\zeta}_{i}(t) \! := \!\displaystyle\int_{0}^{t}\!\bm{\xi}_{i}(s)\:\differential s = \!\displaystyle\int_{0}^{t}\!\!\dfrac{s^{d-i}}{(d-i)!}\:\differential s = \dfrac{t^{d-i+1}}{(d-i+1)!},\]
where $i=1,2,\hdots,d$. Therefore,
\[\|\bm{\zeta}(t)\|_{2} = \left(\sum_{j=1}^{d}\left(t^{j}/j!\right)^{2}\right)^{\!\!1/2},\] 
and the proof is complete.
\end{proof}

\end{document}